\documentclass{amsart}
\newtheorem{theorem}{Theorem}[section]
\newtheorem{lemma}[theorem]{Lemma}

\newtheorem{cor}[theorem]{Corollary}
\newtheorem{definition}[theorem]{Definition}

\newcommand{\ls}{\setbox0\hbox{$-$}
\mathbin{\hbox{$-$\kern-\wd0\raise2\dp0\hbox{$\cdot$}\kern.3\wd0\lower2\dp0\hbox{$\cdot$}}}}
\newcommand{\rs}{\setbox0\hbox{$-$}
\mathbin{\hbox{$-$\kern-\wd0\lower2\dp0\hbox{$\cdot$}\kern.3\wd0\raise2\dp0\hbox{$\cdot$}}}}

\renewcommand{\ln}{\mathord{\sim}}
\newcommand{\rn}{\mathord{-}}
\def\ld{\mathord{\backslash}}
\def\rd{\mathord{/}}
\def\up{\mathord{\uparrow}}

\title[Kites and Pseudo BL-algebras]{Kites and Pseudo BL-algebras}
\date{\today}

\author[A. Dvure\v censkij]{Anatolij Dvure\v censkij$^1$}

\address{$^1$Mathematical Institute, Slovak Academy of Sciences,
\v Stef\'anikova 49, SK--814~73 Bratislava, Slovakia.  Department
of Algebra and Geometry, Faculty of Sciences, Palack\'{y}
University, t\v{r}. 17. listopadu 1192/12, CZ-771 46 Olomouc, Czech
Republic}
\email{dvurecenskij@mat.savba.sk}

\author[T. Kowalski]{Tomasz Kowalski$^2$}

\address{$^2$The University of Melbourne,
Department of Mathematics and Statistics,
Parkville, VIC 3010, Australia}
\email{tomaszk@unimelb.edu.au}

\thanks{The paper has been supported  by the
Center of Excellence SAS -~Quantum Technologies,   meta-QUTE ITMS
26240120022, the grant VEGA No. 2/0059/12 SAV, and by
CZ.1.07/2.3.00/20.0051.}

\subjclass[2010]{Primary 06D35; Secondary 03G12, 03B50}
\keywords{Keywords: MV-algebra, pseudo MV-algebra, pseudo
BL-algebra,  $\ell$-group, unital $\ell$-group, kite, subdirectly irreducible kite, Scrimger
group.}

\begin{document}

\begin{abstract}
We investigate a construction of a pseudo BL-algebra out of an
$\ell$-group called a kite. We show that many well-known examples of algebras
related to fuzzy logics can be obtained in that way. We describe subdirectly irreducible kites. As another
application, we exhibit a new countably infinite family of varieties
of pseudo BL-algebras covering the variety of Boolean algebras.
\end{abstract}

\maketitle

\section{Introduction}

Lattice-ordered groups ($\ell$-groups) are intimately connected with
certain algebras related to fuzzy logics. Indeed, the discovery of
one such connection predates fuzzy logics, as it was made by Chang
\cite{Cha} in his algebraic proof of completeness of
infinitely-valued \L{u}kasiewicz's logic. Much later Mundici
\cite{Mun} proved a categorical equivalence between the variety of
MV-algebras and the class of Abelian $\ell$-groups with strong unit.
This was extended by Dvure\v{c}enskij \cite{Dvu1} to an equivalence
between the variety of \emph{pseudo MV-algebras} (a noncommutative
generalization of MV-algebras) and the class of all $\ell$-groups
with strong unit. For \emph{BL-algebras}, which constitute an
algebraic semantics of classical fuzzy logic, the $\ell$-group
connection was investigated by Aglian\`o and Montagna \cite{AgMo}
who proved that all linearly ordered BL-algebras consist of building blocks that are
either negative cones of $\ell$-groups, or negative intervals in unital
$\ell$-groups. A similar result was proved by Dvure\v{c}enskij for
representable (i.e., such that subdirectly irreducibles are linearly
ordered) pseudo BL-algebras (cf. \cite{Dvu2}). Jipsen and Montagna
\cite{JiMo} constructed a subdirectly irreducible pseudo BL-algebra,
which is not linearly ordered, yet it is made out of an $\ell$-group
in a rather special way, which may be visualized as resembling a
kite: a two-dimensional head joined to a one-dimensional tail. In
this paper we generalize their construction, and show that many
well-known examples of the algebras mentioned above can be seen as
particular cases of the generalized construction. As our
generalization consists essentially in allowing more dimensions for
the head and tail, we call our algebras \emph{kites}.

The paper is organized as follows. Basic notions and notations are in Section 2. Section 3 defines kites, they are always pseudo BL-algebras. It gathers     the main properties of kites. In particular, some kites are pseudo MV-algebras. Section 4 presents a list of important kites. Subdirectly irreducible kites and their classification  are described in Section 5. Very important kites are finite-dimensional ones which are studied in Section 6, where it is shown that the variety of pseudo BL-algebras generated by all kites is generated by all finite-dimensional kites. Finally, Section 7 shows some applications of theory of kites, in particular,  countably many covers generated by kites of the variety of Boolean algebras in the variety of pseudo BL-algebras are presented.

\section{Basic notions and notation}

In terminology and notation we will follow~\cite{GJKO}, which is
also our standard reference for undefined notions and details. All
classes (varieties) of algebras we deal with in the paper can be
viewed as subclasses (subvarieties) of \emph{FL-algebras}, that is,
algebras $\mathbf{A} = \langle
A;\wedge,\vee,\ld,\rd,\cdot,0,1\rangle$ of type $(2,2,2,2,2,0,0)$
such that
\begin{itemize}
\item $\langle A;\wedge,\vee\rangle$ is a lattice,
\item $\langle A;\cdot,\ld,\rd, 1\rangle$ is a residuated monoid,
\item $0$ is an element of $A$.
\end{itemize}
The operations $\ld$ and $\rd$ are called, respectively,
\emph{left division} (or \emph{right residuation}) and
\emph{right division} (or \emph{left residuation}).
Two unary operations: \emph{left negation} $\ln x = x\ld 0$, and
\emph{right negation} $\rn x = 0\rd x$, are commonly used, and will play a
role in the paper.
Negations bind stronger than multiplication, which binds stronger than
divisions, which in turn
bind stronger than the lattice connectives.
The following identities will be of some importance later.
\begin{enumerate}
\item\label{integr} $1\geq x$,
\item\label{0-bound} $0\leq x$,
\item\label{rl} $1 = 0$,
\item\label{divis} $x(x\ld (x\wedge y)) = x\wedge y = ((x\wedge y)\rd x)x$,
\item\label{div-int} $x(x\ld y) = x\wedge y = (y\rd x)x$,
\item\label{prelin} $x\ld y \vee y\ld x = 1 = y\rd x \vee x\rd y$,
\item\label{MV-int} $x\rd(y\ld x) = x\vee y = (x\rd y)\ld x$,
\item\label{MV-gen} $x\rd((x\vee y)\ld x) = x\vee y = (x\rd(x\vee y))\ld x$,
\item\label{dbl-neg} $\ln\rn x = x = \rn\ln x$,
\item\label{good} $\ln\rn x = \rn\ln x$,
\item\label{lg} $1 = x(x\ld 1)$,
\item\label{comm} $xy = yx$.
\end{enumerate}
We write $\mathsf{FL}$ for the variety of all FL-algebras. In
general, if X-algebras are defined and happen to form a variety, we
will use sans-serif $\mathsf{X}$ for that variety. Thus,
$\mathsf{FL_i}$ is the variety of $\text{FL}_i$-algebras: these are
FL-algebras satisfying (\ref{integr}); they are also known as
\emph{integral}. FL-algebras satisfying (\ref{0-bound}) are called
$\text{FL}_o$-algebras, or \emph{zero-bounded}. FL-algebras that are
both integral and zero-bounded are known as $\text{FL}_w$-algebras.
FL-algebras satisfying (\ref{rl}) are \emph{residuated lattices};
this trick allows us to make the constant $0$ disappear; it is
especially useful for viewing lattice-ordered groups ($\ell$-groups)
as a variety of FL-algebras.

The identities (\ref{divis}) and (\ref{div-int}) are both known as
\emph{divisibility}; in varieties of integral FL-algebras
(\ref{divis}) is equivalent to (\ref{div-int}), but not so in
general. FL-algebras satisfying divisibility in the form
(\ref{divis}) are \emph{GBL-algebras}; integral GBL-algebras are
typically defined using (\ref{div-int}) instead. This has the
advantage of making lattice meet definable by means of (either of)
the divisions. As integrality is derivable from (\ref{div-int}), it
is a natural and economical choice of an equational base. For more
on $\mathsf{GBL}$ we refer the reader to~\cite{GaTs,JiMo}. The
identities in (\ref{prelin}) are together known as
\emph{prelinearity}. $\text{FL}_w$-algebras satisfying prelinearity
and divisibility are \emph{pseudo BL-algebras} (cf.
e.g.,~\cite{DGI1,DGI2}). This variety is of prime importance in this
paper, since the kites from the title are certain pseudo BL-algebras
constructed out of $\ell$-groups in a special way. We use
$\mathsf{psBL}$ as a name for the variety of pseudo BL-algebras.

An important subvariety of $\mathsf{psBL}$ is the variety
$\mathsf{psMV}$ of \emph{pseudo MV-algebras}: these are pseudo
BL-algebras satisfying (\ref{MV-int}), or,
equivalently\footnote{Over pseudo BL-algebras; the equivalence  does
not hold in general.}, pseudo BL-algebras satisfying
(\ref{dbl-neg}). Pseudo MV-algebras are categorically equivalent to
the class of $\ell$-groups with a strong unit, as shown in
\cite{Dvu1}. The role of (\ref{MV-gen}) with respect to
(\ref{MV-int}) is analogous to that of (\ref{divis}) with respect to
(\ref{div-int}), namely, (\ref{MV-gen}) it is a non-integral version
of (\ref{MV-int}). In particular, (\ref{MV-gen}) holds in
$\ell$-groups, while (\ref{MV-int}) does not. For this to make
sense, we need to interpret $\ell$-groups as residuated lattices,
and this is done with the help of (\ref{lg}). Namely, the subvariety
of residuated lattices satisfying (\ref{lg}) is term equivalent to
the variety of $\ell$-groups, upon defining $x^{-1} = x\ld 1$ one
way, and $x\ld y = x^{-1}y$, $x\rd y = xy^{-1}$ the other.

Pseudo BL-algebras satisfying (\ref{good}), a natural weakening of
(\ref{dbl-neg}), are known by an unassuming name of \emph{good}. For
example, every pseudo MV-algebra is good. It was an open question
for some time whether every pseudo BL-algebra was good
(cf.~\cite[Problem 3.21]{DGI2}). It was resolved in the negative
in~\cite{DGK}, and the counterexample was in fact a special type of
what we now call a kite, defined in~\cite{JiMo}.

Finally, commutative pseudo BL-algebras are just \emph{BL-algebras}
and commutative pseudo MV-algebras are \emph{MV-algebras}. This
reflects the order of discovery: ``pseudo'' varieties\footnote{Not
pseudovarieties, which are classes closed
  under finite direct products, subalgebras and homomorphic images.}
were discovered as noncommutative generalizations of BL- and
MV-algebras, respectively.

Let $\mathbf{A}$ be an FL-algebra.  Given $a \in A,$ we define
$a^1:=a$ and $a^{n+1}:=a^na,$ for $n \ge 1.$ An element $a\in A$ is
said to be (i) \emph{idempotent} if $a^2 = a$, (ii) \emph{Boolean}
if it is idempotent and $\rn\ln a=a=\ln\rn a$. Let $B(\mathbf A)$ be
the set of Boolean elements of $\mathbf A.$  It is the greatest Boolean subalgebra
of $\mathbf A.$ If $\mathbf A$ is a BL-algebra, by \cite[Prop
2.10]{DGI2}, an element $a\in A$ is Boolean if{}f $a \vee \ln a =1$
if{}f $a\vee \rn a$. Then $\rn a= \ln a$.

We now recall some basic facts about
arithmetical and structural properties of FL-algebras, residuated lattices and
pseudo BL-algebras that we wish to use freely later.

\begin{lemma}\label{arith}
Let $\mathbf{A}$ be an FL-algebra. Then the equivalences
$$
x\leq z\rd y \text{\quad if{}f\quad} xy\leq z \text{\quad
if{}f\quad} y\leq x\ld z
$$
hold for any $x,y,z\in A$. If $\mathbf{A}$ is integral, then $xy
\leq x\wedge y$ always holds, so in particular $x^2\leq x$; moreover
$x\ld y = 1 = y\rd x$ hold if{}f  $x\leq y$. If $\mathbf{A}$ is
divisible and $x\geq y$, there exist elements $z_1$ and $z_2$ such
that $xz_1 = y = z_2x$; in fact $z_1 = x\ld y$ and $z_2 = y\rd x$.
If $\mathbf{A}$ is divisible and $z$ is idempotent, then $xz =
x\wedge z = zx$ holds for any $x\in A$.
\end{lemma}

Let $\mathbf{A}$ be an FL-algebra. A \emph{left conjugate} of an
element $x\in A$ by an element $y\in A$ is the element $y\ld
xy\wedge 1$, and its \emph{right conjugate} is $yx\rd y\wedge 1$. An
element $x\in A$ is \emph{central} if $yx = xy$ holds for all $y\in
A$. By the last statement of Lemma~\ref{arith} all idempotent
elements in pseudo BL-algebras are central. A \emph{filter} of
$\mathbf{A}$ is a set $F\subseteq A$ such that (i) $F$ is a
subalgebra of $0$-free reduct of $\mathbf{A}$ and (ii) $F$ is convex
as an ordered set.   A filter $F$ is normal if it is closed under
conjugates, i.e., for all $x\in F$ and all $y\in A$ both $y\ld xy$
and $yx\rd y$ belong to $F$. Normal filters are also called
\emph{convex normal subalgebras} (e.g. in~\cite{GJKO}), which makes
good sense for residuated lattices but can be confusing for
FL-algebras: convex normal subalgebras may not be subalgebras in the
proper sense because they do not need to contain $0$. If
$\mathbf{A}$ is integral, filters can be alternatively defined as
subsets of $A$ that are upward closed and closed under
multiplication. For a set $S\subseteq A$ we denote its upward
closure by $\up S$. If $S = \{s\}$ we write $\up s$ instead of
$\up\{s\}$.

\begin{lemma}\label{cong}
Let $\mathbf{A}$ be an FL-algebra. For each congruence $\theta$ on $\mathbf{A}$,
its class $1/\theta$ is a normal filter. Conversely, each normal filter $F$
corresponds to a unique congruence $\theta_F$ on $\mathbf{A}$. This
correspondence establishes a lattice isomorphism
between the lattices of normal filters of $\mathbf{A}$ and
congruences of $\mathbf{A}$. If $\mathbf{A}$ is integral and $z\in A$ is
idempotent and central, then $\up z$ is a normal filter. If $\mathbf{A}$ is,
moreover, divisible and $y\in A$ is idempotent, then $\up y$ is normal.
\end{lemma}




\section{Kites}

Let $\mathbf{G}$ be an $\ell$-group, and $I$, $J$ be sets with
$|J|\leq |I|$. Since only the cardinalities of $I$ and $J$ matter
for the construction, it is harmless to think of these sets as
ordinals. We will do so explicitly in the next section. Let further
$\lambda,\rho\colon J\to I$ be injections. Now we define an algebra
with the universe $(G^+)^J \uplus (G^-)^I$. We order its universe by
keeping the original coordinatewise ordering within $(G^+)^J$ and
$(G^-)^I$, and setting $x\leq y$ for all $x\in(G^+)^J$,
$y\in(G^-)^I$. It is easy to verify that this is a (bounded) lattice
ordering of $(G^+)^J \uplus (G^-)^I$. Notice also that the case
$I=J$ is not excluded, so the element $e^I$ may appear twice: at the
bottom of $(G^+)^J$ and at the top of $(G^-)^I$. To avoid confusion
in the definitions below, we adopt a convention of writing
$a_i^{-1},b_i^{-1}, \dots$ for elements of $(G^-)^I$ and
$f_j,g_j,\dots$ for elements of $(G^+)^J$. In particular, we will
write $e^{-1}$ for $e$ as an element of $G^-$. We also put $1$ for
the constant sequence $(e^{-1})^I$ and $0$ for the constant sequence
$e^J$. With these conventions in place we are ready to define
multiplication, putting:
\begin{align*}
\langle a_i^{-1}\colon i\in I\rangle\cdot\langle b_i^{-1}\colon i\in I\rangle &=
  \langle(b_ia_i)^{-1}\colon i\in I\rangle\\
\langle a_i^{-1}\colon i\in I\rangle\cdot\langle f_j\colon j\in J\rangle &=
  \langle a_{\lambda(j)}^{-1}f_j\vee e\colon j\in J\rangle\\
\langle f_j\colon j\in J\rangle\cdot\langle a_i^{-1}\colon i\in I\rangle &=
  \langle f_ja_{\rho(j)}^{-1}\vee e\colon j\in J\rangle\\
\langle f_j\colon j\in J\rangle\cdot\langle g_j\colon j\in J\rangle &=
  \langle e\colon j\in J\rangle = 0.\\
\end{align*}

\begin{definition}\label{resid}
Divisions, $\rd$ and $\ld$, corresponding to multiplication defined
as above on $(G^+)^J \uplus (G^-)^I$  are defined by:
\begin{align*}
\langle a_i^{-1}\colon i\in I\rangle\ld\langle b_i^{-1}\colon i\in I\rangle &=
 \langle a_ib_i^{-1}\wedge e^{-1}\colon i\in I\rangle\\
\langle b_i^{-1}\colon i\in I\rangle\rd\langle a_i^{-1}\colon i\in I\rangle &=
 \langle b_i^{-1}a_i\wedge e^{-1}\colon i\in I\rangle\\
\langle a_i^{-1}\colon i\in I\rangle\ld\langle f_j\colon j\in J\rangle &=
 \langle a_{\lambda(j)}f_j\colon j\in J\rangle\\
\langle f_j\colon j\in J\rangle\rd\langle a_i^{-1}\colon i\in I\rangle &=
 \langle f_ja_{\rho(j)}\colon j\in J\rangle\\
\langle f_j\colon j\in J\rangle\ld\langle g_j\colon j\in J\rangle &= \langle a_i^{-1}\colon i\in I\rangle,\\
\text{ where } a_i^{-1} &=\begin{cases}
f_{\rho^{-1}(i)}^{-1}g_{\rho^{-1}(i)}\wedge e^{-1} &
\text{ if } \rho^{-1}(i) \text{ is defined}\\
e^{-1} & \text{ otherwise}
\end{cases}\\
\langle g_j\colon j\in J\rangle\rd\langle f_j\colon j\in J\rangle &= \langle b_i^{-1}\colon i\in I\rangle,\\
\text{ where } b_i^{-1} &=\begin{cases}
g_{\lambda^{-1}(i)}f_{\lambda^{-1}(i)}^{-1}\wedge e^{-1} &
\text{ if } \lambda^{-1}(i) \text{ is defined}\\
e^{-1} & \text{ otherwise},
\end{cases}\\
\langle a_i^{-1}\colon i\in I\rangle\rd \langle f_j\colon j\in
J\rangle &=(e^{-1})^I= \langle f_j\colon j\in J\rangle \ld \langle
a_i^{-1}\colon i\in I\rangle.
\end{align*}
\end{definition}

We will call the algebra we have just defined a \emph{kite} of
$\mathbf{G}$, and write $K_{I,J}^{\lambda,\rho}(\mathbf{G})$ for it.
Observe that if we take $I=J$, then $\lambda$ and $\rho$ become permutations of
the set of coordinates and so the kite construction is reminiscent of
wreath product. This analogy is not mistaken, as we will see later.
For the moment, let us focus on the algebra
$K_{I,J}^{\lambda,\rho}(\mathbf{G})$.
The next lemma shows that the algebra
$K_{I,J}^{\lambda,\rho}(\mathbf{G}) =
\langle (G^+)^J \uplus (G^-)^I;\vee,\wedge,\cdot,\ld,\rd, 1,0\rangle$
is a pseudo BL-algebra.

\begin{lemma}\label{hyper-psBL}
For any $\ell$-group $\mathbf{G}$ and any choice of appropriate sets $I,J$
and maps $\lambda,\rho$, the algebra $K_{I,J}^{\lambda,\rho}(\mathbf{G})$
is a pseudo BL-algebra.
\end{lemma}

\begin{proof}
It is clear that  $K_{I,J}^{\lambda,\rho}(\mathbf{G})$ is a
lattice-ordered groupoid with unit. To show that multiplication is
associative, first observe that triples from $(G^-)^I$ associate
because $(G^-)^I$ is just the negative cone of $\mathbf{G}^I$. Next,
triples involving at least two elements from $(G^+)^J$ associate
because both the products equal $0$. The remaining cases all involve
one element from $(G^-)^I$ and two from $(G^+)^J$. One such case is:
\begin{align*}
(\langle a_i^{-1}\colon i\in I\rangle\cdot\langle f_j\colon j\in J\rangle)\cdot\langle b_i^{-1}\colon i\in I\rangle &=
\langle a_{\lambda(j)}^{-1}f_j\vee e\colon j\in J\rangle\cdot\langle b_i^{-1}\colon i\in I\rangle\\
&= \langle( a_{\lambda(j)}^{-1}f_j\vee e)b_{\rho(j)}^{-1}\vee e\colon j\in J\rangle\\
&= \langle a_{\lambda(j)}^{-1}f_jb_{\rho(j)}^{-1}\vee b_{\rho(j)}^{-1}\vee e\colon j\in J\rangle\\
&= \langle a_{\lambda(j)}^{-1}f_jb_{\rho(j)}^{-1}\vee e\colon j\in J\rangle\\
&= \langle a_{\lambda(j)}^{-1}f_jb_{\rho(j)}^{-1}\vee a_{\lambda(j)}^{-1}\vee e\colon j\in J\rangle\\
&= \langle a_{\lambda(j)}^{-1}(f_jb_{\rho(j)}^{-1}\vee e)\colon j\in J\rangle\\
&= \langle a_i\colon i\in I\rangle\cdot\langle f_jb_{\rho(j)}^{-1}\vee e\colon j\in J\rangle\\
&= \langle a_i^{-1}\colon i\in I\rangle\cdot(\langle f_j\colon j\in J\rangle\cdot\langle b_i^{-1}\colon i\in I\rangle).
\end{align*}
Other cases follow by similar calculations. Now,
to show that divisibility holds, we also proceed case by case. Let us
deal with two cases here. The first is:
\begin{align*}
\langle f_j\colon j\in J\rangle\cdot(\langle f_j\colon j\in J\rangle\ld\langle g_j\colon j\in J\rangle)
&= \langle f_j\colon j\in J\rangle\cdot\langle a_i^{-1}\colon i\in I\rangle \\
&= \langle f_ja_{\rho(j)}^{-1}\vee e\colon j\in J\rangle
\end{align*}
where
$$
a_i^{-1} =\begin{cases}
f_{\rho^{-1}(i)}^{-1}g_{\rho^{-1}(i)}\wedge e^{-1} &
\text{ if } \rho^{-1}(i) \text{ is defined}\\
e^{-1} & \text{ otherwise}
\end{cases}
$$
but, observe that $\rho^{-1}(\rho(j))$ is always defined and equals $j$,
so calculating further we obtain
$f_ja_{\rho(j)}^{-1} = g_j\wedge f_j$ for every $j\in J$, and therefore
$$
\langle f_ja_{\rho(j)}^{-1}\vee e\colon j\in J\rangle =
\langle(f_j\wedge g_j)\vee e\colon j\in J\rangle =
\langle f_j\wedge g_j\colon j\in J\rangle
$$
as required. For the second, take:
\begin{align*}
\langle a_i^{-1}\colon i\in I\rangle\cdot(\langle a_i^{-1}\colon i\in I\rangle\ld\langle f_j\colon j\in J\rangle) &=
\langle a_i^{-1}\colon i\in I\rangle\cdot\langle a_{\lambda(j)}f_j\colon j\in J\rangle\\
&= \langle a_{\lambda(j)}^{-1}a_{\lambda(j)}f_j\vee e\colon j\in J\rangle\\
&= \langle f_j\colon j\in J\rangle.
\end{align*}
All other cases are straightforward. It remains to show prelinearity. Since
multiplication and divisions in $(G^-)^I$ are defined  coordinatewise,
prelinearity for $x,y\in (G^-)^I$ is inherited from $G^-$. If
$x\in (G^-)^I$ and $y\in (G^+)^J$, or \emph{vice versa} prelinearity holds
trivially. For the only remaining case, calculating
$$
\langle f_j\colon j\in J\rangle\ld\langle g_j\colon j\in J\rangle\vee
\langle g_j\colon j\in J\rangle\ld\langle f_j\colon j\in J\rangle
$$
yields two cases: (1) if $\rho^{-1}(i)$ is defined, we have
\begin{align*}
& f_{\rho^{-1}(i)}^{-1}g_{\rho^{-1}(i)}\wedge e^{-1}
\vee  g_{\rho^{-1}(i)}^{-1}f_{\rho^{-1}(i)}\wedge e^{-1}\\
&= (f_{\rho^{-1}(i)}^{-1}g_{\rho^{-1}(i)}\vee
g_{\rho^{-1}(i)}^{-1}f_{\rho^{-1}(i)})\wedge e^{-1})\\
&=  e^{-1}
\end{align*}
and (2) if $\rho^{-1}(i)$ is not defined, we have
$$
a^{-1}_i \vee a^{-1}_i = e^{-1}\vee e^{-1}= e^{-1}
$$
as well. Thus, prelinearity holds and that finishes the proof of all
the claims in the lemma.
\end{proof}

Somewhat surprisingly, many kites turn out to be pseudo MV-algebras.

\begin{lemma}\label{hyper-psMV}
Let $\mathbf{G}$ be an $\ell$-group, and
suppose $|I| = |J|$ and $\lambda,\rho$ are bijections. Then
$K_{I,J}^{\lambda,\rho}(\mathbf{G})$
is a pseudo MV-algebra.
\end{lemma}

\begin{proof}
By Lemma~\ref{hyper-psBL} we only need to show that under the conditions of the
lemma, the identity $x\rd(y\ld x) = x\vee y = (x\rd y)\ld x$ holds.
We have two nontrivial cases to consider:

\medskip
\noindent
\emph{Case 1}. $x\in(G^+)^J$ and $y\in(G^-)^I$. Then,
$x = \langle f_j\colon j\in J\rangle \leq \langle a_i^{-1}\colon i\in I\rangle = y$, and so we calculate:
\begin{align*}
\langle f_j\colon j\in J\rangle\rd(\langle a_i^{-1}\colon i\in I\rangle\ld\langle f_j\colon j\in J\rangle) &=
\langle f_j\colon j\in J\rangle\rd\langle a_{\lambda(j)}f_j\colon j\in J\rangle\\
&= \langle f_{\lambda^{-1}(i)}(a_{\lambda(j)}f_j)^{-1}_{\lambda^{-1}(i)}\wedge e^{-1}\colon i\in I\rangle\\
&= \langle f_{\lambda^{-1}(i)}f_{\lambda^{-1}(i)}^{-1}a_i^{-1}\colon i\in I\rangle\\
&= \langle a_i^{-1}\colon i\in I\rangle
\end{align*}
which shows that $x\rd(y\ld x) = y = x\vee y$ holds. Notice that we used
bijectiveness of $\lambda$ to pass from the first to the second equality above.

\medskip
\noindent
\emph{Case 2}. $x,y\in(G^+)^J$. Then,
$x = \langle f_j\colon j\in J\rangle$ and $y = \langle g_j\colon j\in J\rangle$,
and so we calculate:
\begin{align*}
\langle f_j\colon j\in J\rangle\rd(\langle g_j\colon j\in J\rangle\ld\langle f_j\colon j\in J\rangle) &=
\langle f_j\colon j\in J\rangle\rd\langle g_{\rho^{-1}(i)}^{-1}f_{\rho^{-1}(i)}\wedge e^{-1}\colon i\in I\rangle\\
&= \langle f_j(g_{\rho^{-1}(i)}^{-1}f_{\rho^{-1}(i)}\wedge e^{-1})^{-1}_{\rho(j)}\colon j\in J\rangle\\
&= \langle f_j(f_j^{-1}g_j\vee e)\colon j\in J\rangle\\
&= \langle g_j\vee f_j\colon j\in J\rangle
\end{align*}
which again shows that $x\rd(y\ld x) = x\vee y$ holds. The proofs for
$x\vee y = (x\rd y)\ld x$ are symmetric.
\end{proof}

As we already mentioned, it was an open problem (\cite[Problem
3.21]{DGI2}) for a while whether every pseudo BL-algebra is good. We
found a negative solution in \cite{DGK} using special types of kites
from~\cite{JiMo}. Below we characterize good kites, and in the last
section we will exhibit a countably infinite family of varieties of
kites with the property that their all and only good members are
Boolean algebras.

\begin{lemma}\label{good-kites}
Let $\mathbf{G}$ be an $\ell$-group, and
$K_{I,J}^{\lambda,\rho}(\mathbf{G})$ a kite.
\begin{enumerate}
\item
$K_{I,J}^{\lambda,\rho}(\mathbf{G})$ is good if and only if
$\lambda(J)=\rho(J)$.
\item $K_{I,J}^{\lambda,\rho}(\mathbf{G})$ is a pseudo MV-algebra
if and only if $\lambda(J)=I=\rho(J)$.
\end{enumerate}
\end{lemma}

\begin{proof}
(1) Let $x= \langle a_i\colon i\in I\rangle.$  Then $\rn x = \langle
a_{\rho(j)}\colon j\in J\rangle$ and $\ln x= \langle
a_{\lambda(j)}\colon j\in J\rangle.$ In addition, $\ln\rn x=
\langle z_i\colon i\in I\rangle, $ where
$$
z_i^{-1}=\begin{cases} a_i^{-1} &
\text{ if } \lambda^{-1}(i) \text{ is defined}\\
e^{-1} & \text{ otherwise},
\end{cases}
$$
and $\rn\ln x= \langle y_i\colon i\in I\rangle$, where
$$
y_i^{-1}=\begin{cases} a_i^{-1} &
\text{ if } \rho^{-1}(i) \text{ is defined}\\
e^{-1} & \text{ otherwise}.
\end{cases}
$$
Now, if $x = \langle f_j\colon j\in J\rangle$, we have
$\ln\rn x = \langle g_j\colon j\in J\rangle$, where
$$
g_j =
\begin{cases} f_j & \text{ if } \rho^{-1}(i) \text{ is defined}\\
e & \text{ otherwise},
\end{cases}
$$
and $\rn\ln x = \langle h_j\colon j\in J\rangle$, where
$$
h_j =
\begin{cases} f_j & \text{ if } \lambda^{-1}(i) \text{ is defined}\\
e & \text{ otherwise}.
\end{cases}
$$
Hence, if $\lambda(J)=\rho(J),$ the kite
$K_{I,J}^{\lambda,\rho}(\mathbf{G})$ is good.

Conversely, assume that the kite
$K_{I,J}^{\lambda,\rho}(\mathbf{G})$ is good, and let $\lambda(J)\ne
\rho(J).$ For  our aims we can assume that each $a_i^{-1} \ne
e^{-1}.$ There is an $i\in I$ such that either $\lambda^{-1}(i)$ or
$\rho^{-1}(i)$ is not defined. Equivalently, $x_i^{-1}= e^{-1}$ and
$x_i^{-1}=a_i^{-1}$ or $y_i^{-1}=e^{-1}$ and $y_i^{-1}=a_i^{-1}.$
Hence, $\lambda(J)=\rho(J).$

(2) If $\lambda(J)=I=\rho(J)$, then
$K_{I,J}^{\lambda,\rho}(\mathbf{G})$ is a pseudo MV-algebra by Lemma
\ref{hyper-psMV}. Conversely, let the kite be a pseudo MV-algebra.
Since every pseudo MV-algebra is good, by the first part of the
present proof, we have $\lambda(J)=\rho(J)$. Now assume by absurd
that there is an $i\in I\setminus \lambda(J)$. Then both
$\lambda^{-1}(i)$ and $\rho^{-1}(i)$ are not defined and whence
$x_i^{-1}=e^{-1} = y_i^{-1} \ne a_i$ which contradicts the property
$\ln\rn x=x=\rn\ln x$. Whence $\lambda(J)=I=\rho(J)$.
\end{proof}

\subsection{One-dimensional elements}
One property of kites that is frequently used in calculations is
that double negations amount to certain shifts (often, rotations) of
coordinate system. Below, we will state this observation in the form
that will later help characterize finite-dimensional kites.

Let $K_{I,J}^{\lambda,\rho}(\mathbf{G})$ be a kite. An element
$a\in (G^-)^I$  will be called $\alpha$-dimensional, for some
cardinal $\alpha$, if $|\{i\in I\colon a(i) \neq e\}| = \alpha$.
Similarly, this notion applies to elements of $(G^+)^J$.
One-dimensional elements are particularly easy to work with, and, moreover,
it is immediately seen that every element of a kite is a join or a meet of
one-dimensional elements.

\begin{lemma}\label{rotations}
Let $K_{I,J}^{\lambda,\rho}(\mathbf{G})$ be a kite and $a\in
(G^-)^I$ be one-dimensional, with $a(i)\neq e^{-1}$. Then $\ln\ln
a$, $\ln\rn a$, $\rn\ln a$, and $\rn\rn a$ are at most
one-dimensional. If $\lambda^{-1}(i)$ and $\rho^{-1}(i)$ are
defined, these elements are exactly one-dimensional. More precisely,
the following hold:
\begin{enumerate}
\item $\ln\ln a< 1$ and $\rn\ln a<1$ if{}f $\lambda^{-1}(i)$ is defined,
\item $\rn\rn a< 1$ and $\ln\rn a<1$ if{}f $\rho^{-1}(i)$ is defined,
\item $\rn\ln a = a$ if{}f $\lambda^{-1}(i)$ is defined,
\item $\ln\rn a = a$ if{}f $\rho^{-1}(i)$ is defined,
\item $\ln\ln a\vee a = 1$ if{}f $\rho(\lambda^{-1}(i))\neq i$,
\item $\rn\rn a\vee a = 1$ if{}f $\lambda(\rho^{-1}(i))\neq i$.
\end{enumerate}
\end{lemma}

\begin{proof}
To begin with, since $\ln\rn a \geq a$ and $\rn\ln a\geq a$ hold in
any residuated lattice, $\ln\rn a$ and $\rn\ln a$ are at most
one-dimensional. Let us calculate $\ln\ln a$. To make the notation
less cumbersome, we first regard $a$ as a sequence $\langle
e^{-1},\dots, a^{-1}(i),e^{-1},\dots\rangle$, and then write
$\langle a^{-1}(i)\rangle$ for $\langle
e^{-1},\dots,a^{-1}(i),e^{-1},\dots\rangle$ and $\langle
a(i)\rangle$ for $\langle e,\dots, a(i), e,\dots\rangle$. Similarly,
we put $\langle e\rangle$ for $e^J = 0$ and $\langle e^{-1}\rangle$
for $(e^{-1})^I = 1$. Calculating $\langle
a^{-1}(i)\rangle\ld\langle e\rangle$ we see that it is different
from $e$ at coordinate $\lambda(j)$ if and only if $\lambda(j) = i$.
Therefore
$$
\langle a^{-1}(i)\rangle\ld\langle e\rangle =\begin{cases}
\langle a(\lambda^{-1}(i))\rangle &
\text{ if } \lambda^{-1}(i) \text{ is defined}\\
\langle e\rangle & \text{ otherwise}.
\end{cases}
$$
Then, calculating $\langle a(\lambda^{-1}(i))\rangle\ld\langle e\rangle$ in
turn, we get that it is different from $e^{-1}$ at
coordinate $k$ if and only if $\rho^{-1}(k)$ is defined and
equal to $\lambda^{-1}(i)$, that is
if and only if $k = \rho(\lambda^{-1}(i))$. Altogether, we have
$$
\bigl(\langle a^{-1}(i)\rangle\ld\langle e\rangle\bigr)\ld\langle e\rangle =\begin{cases}
\langle a^{-1}(\rho\circ\lambda^{-1}(i))\rangle &
\text{ if } \lambda^{-1}(i) \text{ is defined}\\
\langle e^{-1}\rangle & \text{ otherwise}.
\end{cases}
$$
By similar calculations we also obtain
\begin{align*}
\langle e\rangle\rd\bigl(\langle a^{-1}(i)\rangle\ld\langle e\rangle\bigr) &=\begin{cases}
\langle a^{-1}(i)\rangle &
\text{ if } \lambda^{-1}(i) \text{ is defined}\\
\langle e^{-1}\rangle & \text{ otherwise}
\end{cases}\\
\bigl(\langle e\rangle\rd\langle a^{-1}(i)\rangle\bigr)\ld\langle e\rangle &=\begin{cases}
\langle a^{-1}(i)\rangle &
\text{ if } \rho^{-1}(i) \text{ is defined}\\
\langle e^{-1}\rangle & \text{ otherwise}
\end{cases}\\
\langle e\rangle\rd\bigl(\langle e\rangle\rd\langle a^{-1}(i)\rangle\bigr) &=\begin{cases}
\langle a^{-1}(\lambda\circ\rho^{-1}(i))\rangle &
\text{ if } \rho^{-1}(i) \text{ is defined}\\
\langle e^{-1}\rangle & \text{ otherwise}.
\end{cases}
\end{align*}
Then all the claims follow from these calculations and
coordinatewise ordering of $(G^-)^I$. To give just one example, for
(5) we have $\ln\ln a\vee a = 1$ if and only if $\langle
a^{-1}(\rho\circ\lambda^{-1}(i))\rangle\vee \langle a^{-1}(i)\rangle
= \langle e^{-1}\rangle$ if and only if
$\rho\circ\lambda^{-1}(i)\neq i$.
\end{proof}

\section{Examples of kites}

As we have already remarked, in the kite construction we can think of the index
sets $I$ and $J$ as ordinals. Doing so systematically also makes
classification of kites easier, so throughout this section we
assume that $I$ and $J$ are ordinals. Below we give five examples of rather
familiar algebras that can be rendered as kites.

\subsection{Boolean algebras}
Let $I = 0 = J$. Then, $(G^-)^I$ and $(G^+)^J$ are both singletons,
and $\lambda = \rho$ can only be the empty function (hence, an
injection). Thus, $K_{0,0}^{\emptyset,\emptyset}(\mathbf{G})$ is the
two-element Boolean algebra for any $\ell$-group $\mathbf{G}$. We
also get the two-element Boolean algebra in another way. Namely, let
$\mathbf{O}$ be the trivial $\ell$-group. Then $\lambda = id =
\rho$, and $K_{I,J}^{id,id}(\mathbf{O})$ is the two-element Boolean
algebra for any choice of $I$ and $J$.

\subsection{Product logic algebras}
Let $I = 1$ and $J=0$. As the only function from $J$ to $I$ is the
empty function (which is an injection),  the kite
$K_{1,0}^{\emptyset,\emptyset}(\mathbf{G})$ is well-defined for any
$\ell$-group $\mathbf{G}$. Let $\mathbb{R}^+$ stand for the
$\ell$-group of positive reals under multiplication. Then
$K_{1,0}^{\emptyset,\emptyset}(\mathbb{R}^+)$ is isomorphic to the
\emph{standard product logic algebra}, i.e., the real interval
$[0,1]$ with usual multiplication, and divisions given by $x\rd y =
\frac{x}{y} = y\ld x$ for $y\neq 0$ and $x\rd 0 = 1 = 0\ld x$.
Taking $\mathbb{Z}$ for $\mathbf{G}$, we obtain as
$K_{1,0}^{\emptyset,\emptyset}(\mathbb{Z})$ the algebra
$\mathbb{Z}^-_\bot =
\langle\mathbb{Z}^-\cup\{\bot\};\mathrm{max},\mathrm{min},
+,-,\bot,0\rangle$, where $\bot = -\infty$. This is also a product
logic algebra. Both $K_{1,0}^{\emptyset,\emptyset}(\mathbb{R}^+)$
and $K_{1,0}^{\emptyset,\emptyset}(\mathbb{Z})$ generate the whole
variety of product logic algebras. This variety covers the variety
of Boolean algebras.

\subsection{Jipsen-Montagna algebras}
Taking $I = 2$, $J =1$, $\lambda(0) = 0$, and $\rho(0) = 1$, we get
that $K_{2,1}^{\lambda,\rho}(\mathbf{G})$ is a Jipsen-Montagna
algebra, for any $\ell$-group $\mathbf{G}$. In particular, the
algebra $K_{2,1}^{\lambda,\rho}(\mathbb{Z})$ is a pseudo BL-algebra which is not good \cite{DGK} and it generates another cover
of the variety of Boolean algebras. In Section~\ref{appl} we will
show that the same holds for $K_{n+1,n}^{\lambda,\rho}(\mathbb{Z})$
with $\lambda(i) = i$ and $\rho(i) = i+1$, for an arbitrary $n\in
\omega$.

\subsection{Chang chain and its subdirect powers}
Now, let $I = 1 = J$. Then $K_{1,1}^{id,id}(\mathbb{Z})$ is the
Chang chain, denoted $\mathbf{S}^\omega_1$ in Komori's paper, and
$\mathbf{C}_\infty$ in \cite{GJKO}. The variety generated by
$K_{1,1}^{id,id}(\mathbb{Z})$ is also a cover of the variety of
Boolean algebras. For $I = J > 1$, the kite
$K_{I,I}^{id,id}(\mathbb{Z})$ is subdirectly embeddable in
$(\mathbf{C}_\infty)^I$ and thus $\mathsf V(K_{1,1}^{id,id}(\mathbb{Z})) =
\mathsf V(K_{I,I}^{id,id}(\mathbb{Z}))$.

\subsection{Intervals in Scrimger groups}
Taking $I = J = n$ for $n\geq 2$, and putting $\lambda(i) = i$ and
$\rho(i) = i+1$ ($\mathrm{mod}\ n$), we get that
$K_{I,I}^{\lambda,\rho}(\mathbb{Z})$ is isomorphic to
$\Gamma(\mathbf{G}_n, (\langle 0\rangle,1))$, where $\mathbf{G}_n$
is the subgroup of $\mathbb{Z}\wr\mathbb{Z}$ (antilexicographically
ordered), consisting of the elements $\langle\langle a_i\colon
i\in\mathbb{Z}\rangle, b\rangle$, such that $i=j$ ($\mathrm{mod}\
n$) implies $a_i = a_j$.

\section{Subdirectly irreducible kites}

We will characterise subdirectly irreducible kites and show they fall into five
broad classes. To begin with, we state a few facts on normal filters in kites.

\begin{lemma}\label{up-max-norm}
Let $K_{I,J}^{\lambda,\rho}(\mathbf{G})$ be a kite. Then
$(G^-)^I$ is a maximal normal filter of $K_{I,J}^{\lambda,\rho}(\mathbf{G})$.
\end{lemma}

\begin{proof}
It is clear that $(G^-)^I$ is a maximal filter, so we only need to show that
$(G^-)^I$ is closed under conjugation by elements outside $(G^-)^I$.
Let $x\in (G^-)^I$ and $y\in(G^+)^J$. Taking $y\ld xy$ we observe that
$xy\in(G^+)^J$. By definition of divisions, we then get that
$y\ld xy\in(G^-)^I$. By symmetry the same holds for right conjugates.
\end{proof}

\begin{lemma}\label{hyper-filters}
Let $\mathbf{G}$ be an $\ell$-group, and
$K_{I,J}^{\lambda,\rho}(\mathbf{G})$ a kite.
Let $F$ be a convex normal subgroup of $\mathbf{G}$, and
$N = F|_{G^-}$. Then $N^I =
\{\langle a^{-1}_i\colon i\in I\rangle\colon a^{-1}_i\in N\}$ is a normal
filter of $K_{I,J}^{\lambda,\rho}(\mathbf{G})$.
\end{lemma}

\begin{proof}
Since multiplication and order are defined coordinatewise on $(G^-)^I$, it is
obvious that $N^I$ is a filter. We need to show that $N^I$ is closed under
conjugates. Now, conjugation by an element of $(G^-)^I$ also proceeds
coordinatewise, so $N^I$ is closed under all such conjugates. Take an element
$\langle f_j\colon j\in J\rangle$ and consider
$\langle f_j\colon j\in J\rangle\ld
\langle a_i^{-1}\colon i\in I\rangle\cdot\langle f_j\colon j\in J\rangle$, for
some $\langle a_i^{-1}\colon i\in I\rangle\in N^I$. This is equal to
$\langle f_j\colon j\in J\rangle\ld
\langle a_{\lambda(j)}^{-1}f_j\vee e\colon j\in J\rangle$, which in turn equals
$\langle b_i^{-1}\colon i\in I\rangle$, where
$$
b_i^{-1} =\begin{cases}
\bigl(f_{\rho^{-1}(i)}^{-1}a_{\lambda\circ\rho^{-1}(i)}^{-1}f_{\rho^{-1}(i)}\wedge
e^{-1}\bigr) \vee f_{\rho^{-1}(i)}^{-1} &
\text{ if } \rho^{-1}(i) \text{ is defined}\\
e^{-1} & \text{ otherwise}
\end{cases}\eqno(0)\\
$$
as one can verify by a series of simple calculations. Notice that
$f_{\rho^{-1}(i)}^{-1}a_{\lambda\circ\rho^{-1}(i)}^{-1}f_{\rho^{-1}(i)}$
is a conjugate of a member of $N$ and thus belongs to $F$ by
normality. Therefore
$f_{\rho^{-1}(i)}^{-1}a_{\lambda\circ\rho^{-1}(i)}^{-1}f_{\rho^{-1}(i)}\wedge
e^{-1}$ belongs to $N$. Thus, by upward closedness of $N$ we get
that $b_i\in N$, for each $i\in I$, and so $\langle b_i^{-1}\colon
i\in I\rangle\in N^I$. This shows that $N^I$ is normal, as required.
\end{proof}

The converse of Lemma~\ref{hyper-filters} is also true, but it will
be useful to define a technical notion before we show it. Namely,
for a proper normal filter $N$ of a kite
$K_{I,J}^{\lambda,\rho}(\mathbf{G})$, and a $k\in I$, we define the
restriction of $N$ to a coordinate $k$ as the set of elements $a\in
N$ with $a(i) \neq e^{-1}$ if{}f $i = k$.

\begin{lemma}\label{restrict}
Let $\mathbf{G}$ be an $\ell$-group,
$K_{I,J}^{\lambda,\rho}(\mathbf{G})$ a kite, and $N$ a normal filter of
$K_{I,J}^{\lambda,\rho}(\mathbf{G})$. For any $k\in I,$ the
restriction of $N$ to $k$ is a normal filter of $\mathbf{G}^-$, and therefore
the negative part of a convex normal subgroup of $\mathbf{G}$.
\end{lemma}

\begin{proof}
Since $N$ is proper, it is contained in $(G^-)^I$, which is
a direct power of $\mathbf{G}^-$. Then, the restriction of $N$ to $k$ is just
the $k$-th projection of $N$, and the claim follows.
\end{proof}

The next lemma shows that subdirectly irreducible kites can only
arise out of subdirectly irreducible $\ell$-groups, and a first
characterization of subdirectly irreducible kites follows.

\begin{lemma}\label{si-only-if-si}
Let $\mathbf{G}$ be an $\ell$-group, and $K_{I,J}^{\lambda,\rho}(\mathbf{G})$ a
kite. If $K_{I,J}^{\lambda,\rho}(\mathbf{G})$ is subdirectly irreducible, so is
$\mathbf{G}$.
\end{lemma}

\begin{proof}
We will prove the contrapositive. Suppose $\mathbf{G}$ is not
subdirectly irreducible. Then, there exists a set
$\{\mathbf{M}_s\colon s\in S\}$ of nontrivial convex normal
subgroups of $\mathbf{G}$, such that $\bigcap_{s\in S}M_s = \{e\}$.
Let $M^-_s = M_s|_{G^-}$. Then, by Lemma~\ref{hyper-filters},
$(M^-_s)^I$ is a normal filter of
$K_{I,J}^{\lambda,\rho}(\mathbf{G})$, for each $s\in S$. Suppose
$\langle a_i\colon i\in I\rangle$ belongs to $(M^-_s)^I$ for each
$s\in S$. Then, for any coordinate $k\in I$ we have that $a_k\in
M^-_s$ for all $s\in S$, and thus $a_k = e$. Therefore,
$\bigcap_{s\in S}(M^-_s)^I = \{e^I\}$, showing that the set
$\{(M^-_s)^I\colon s\in S\}$ of nontrivial normal filters of
$K_{I,J}^{\lambda,\rho}(\mathbf{G})$ intersects trivially. Thus,
$K_{I,J}^{\lambda,\rho}(\mathbf{G})$ is not subdirectly irreducible,
proving the claim.
\end{proof}

\begin{theorem}\label{si-kites}
Let $\mathbf{G}$ be an $\ell$-group, and
$K_{I,J}^{\lambda,\rho}(\mathbf{G})$ a kite.
The following are equivalent:
\begin{enumerate}
\item $\mathbf{G}$ is subdirectly irreducible and for all $i,j\in I$ there exists
  $m\in\omega$ such that
$(\rho\circ\lambda^{-1})^m(i) = j$ or
$(\lambda\circ\rho^{-1})^m(i) = j$.
\item $K_{I,J}^{\lambda,\rho}(\mathbf{G})$ is subdirectly irreducible.
\end{enumerate}
\end{theorem}

\begin{proof}
To prove that (1) implies (2), let $\mathbf{M}$ be the smallest
nontrivial convex normal subgroup of $\mathbf{G}$, and $N$ be the
restriction of $M$ to the negative cone of $\mathbf{G}$. Then, $N$
is the smallest nontrivial normal filter of $\mathbf{G}^-$. By
Theorem~\ref{hyper-filters} we have that $N^I$ is a normal filter of
$K_{I,J}^{\lambda,\rho}(\mathbf{G})$, so it remains to prove that
$N^I$ is the smallest nontrivial such. Since for any $a\in
N^I\setminus\{1\}$ there is a one-dimensional element $a'$ with
$a\leq a'<1$, it suffices to prove that any one-dimensional element
$b\in N^I\setminus\{1\}$ generates $N^I$. Without loss of generality
assume $b = \langle b^{-1}_0,e^{-1},\dots\rangle$; this is always
achievable by a suitable re-ordering of $I$, regardless of its
cardinality. Observe that $b^{-1}_0$ generates $N$, since $N$ is the
smallest nontrivial normal filter of $\mathbf{G}^-$. It follows that
$b$ generates all members of $N^I$ of the form $\langle a^{-1},
e^{-1},\dots\rangle$, using only conjugates of the same form.
Consider an arbitrary $i\in I$. By assumption, there is $m\in\omega$
with $\rho(\lambda^{-1})^m(0) = i$ or $\lambda(\rho^{-1})^m(0) =
i$. Now, repeating $m$ times the calculation from the proof of
Lemma~\ref{rotations}, we obtain that for an element $u = \langle
a^{-1}, e^{-1},\dots\rangle$, one of the following must be the case:
\begin{itemize}
\item if $\rho(\lambda^{-1})^m(0) = i$, then
$\underbrace{\ln\dots\ln}_{\text{$2m$-times}}u= \langle
e^{-1},\dots,e^{-1},a^{-1},e^{-1},\dots\rangle$,
\item if $\lambda(\rho^{-1})^m(0) = i$, then
$\underbrace{\rn\dots\rn}_{\text{$2m$-times}}u = \langle
e^{-1},\dots,e^{-1},a^{-1},e^{-1},\dots\rangle$,
\end{itemize}
where $a^{-1}$ occurs at a coordinate $i$; again by a suitable
renumbering of $I$ we can assume it to be the $m$-th coordinate. By
taking appropriate meets it then follows that every element of $N^I$
can be generated, which proves the claim.

For the the converse, by Lemma~\ref{si-only-if-si} we can assume
$\mathbf{G}$ is subdirectly irreducible. Then, suppose there are
$i,j\in I$ such that for all $m\in \omega$ we have
$\rho(\lambda^{-1})^m(i) \neq j$ and $\lambda(\rho^{-1})^m(i) \neq
j$. We will call such $i$ and $j$ \emph{disconnected}; otherwise,
$i$ and $j$ will be called \emph{connected}. If all distinct members
of a $K\subseteq I$ are connected, we will call $K$ a
\emph{connected component} of $I$. Now, let $I_0$ and $I_1$ be
connected components of $I$ such that $i\in I_0$ and $j\in I_1$.
Clearly, $I_0$ and $I_1$ are disconnected, that is no member of
$I_0$ is connected to any member of $I_1$. We will prove that
$N^{I_0}\cap N^{I_1} = \{1\}$, from which it follows immediately
that $K_{I,J}^{\lambda,\rho}(\mathbf{G})$ is not subdirectly
irreducible. In fact, it suffices to show that for an element $u =
\langle a^{-1}_i\colon i\in I\rangle$ such that $a^{-1}_i = e^{-1}$
for all $i\notin I_0$, and for any element $b$, the conjugate $b\ld
ub$ has $b\ld ub(i) = e^{-1}$ if $i\notin I_0$, and the same holds
for $bu\rd b$. Take $b = \langle b_j\colon j\in J\rangle$. We have
\begin{align*}
b\ld ub &=
\langle b_j\colon j\in J\rangle\ld
\langle a^{-1}_i\colon i\in I\rangle\cdot
\langle b_j\colon j\in J\rangle\\
&= \langle b_j\colon j\in J\rangle\ld
\langle a_{\lambda(j)}^{-1}b_j\vee e\colon j\in J\rangle\\
&= \langle c_i^{-1}\colon i\in I\rangle
\end{align*}
where  $c_i^{-1} =\begin{cases}
b_{\rho^{-1}(i)}^{-1}
(a_{\lambda(\rho^{-1}(i))}b_{\rho^{-1}(i)}\vee e)
\wedge e^{-1} &
\text{ if } \rho^{-1}(i) \text{ is defined}\\
e^{-1} & \text{ otherwise.}
\end{cases}$\\
Now, by assumption $a^{-1}_i = e^{-1}$ for $i\notin I_0$, and
by connectedness, $\lambda(\rho^{-1}(i))\notin I_0$ if $i\notin I_0$.
Therefore, $c_i^{-1}$ can be different from $e^{-1}$ only if
$i\in I_0$, and thus $b\ld ub$ is of the required form. The claim for the other
conjugate follows by symmetry.
\end{proof}

If $I$ and $J$ are finite, the form of subdirectly irreducible kites is even
more restricted than Theorem~\ref{si-kites} explicitly states. In such a case,
$I$ can only be the same size as $J$ or bigger by one,
and $\lambda$ and $\rho$ are essentially determined by the sizes
of $I$ and $J$.

\begin{lemma}\label{si-sizes}
If $\mathbf{G}$ is an $\ell$-group,
$K_{I,J}^{\lambda,\rho}(\mathbf{G})$ a subdirectly irreducible kite,
and $I$ and $J$ are finite, then $K_{I,J}^{\lambda,\rho}(\mathbf{G})$
is isomorphic to one of:
\begin{enumerate}
\item $K_{n,n}^{\lambda,\rho}(\mathbf{G})$, with $\lambda(j) = j$ and
$\rho(j) = j+1\ (\mathrm{mod}\ n)$,
\item $K_{n+1,n}^{\lambda,\rho}(\mathbf{G})$, with $\lambda(j) = j$ and
$\rho(j) = j+1$.
\end{enumerate}
\end{lemma}

\begin{proof}
Assume first that $|I| = n = |J|$. Then, since $\lambda$ and $\rho$
are injections, we can number the elements of $I$ and $J$ so that
$\lambda(j) = j$. If $\rho(k) = k$ for some $k$, then $\{k\}$ is a
connected component of $I$ and thus, by Theorem~\ref{si-kites},
$K_{I,J}^{\lambda,\rho}(\mathbf{G})$ is not subdirectly irreducible,
contradicting the assumption. So, $\rho(j) \neq j$, for every $j$.
We can then renumber $I$ and $J$ so that $\rho(j) = j+1\
(\mathrm{mod}\ n)$.

For the second part, assume $|J|<|I|$. Then, there is $i\in I$
not in the range of $\rho$. We start numbering $I$ by putting $i = 0$.
By Theorem~\ref{si-kites} we get that
$I$ is connected, and so $0\in I$ is in the range of $\lambda$.
We start numbering $J$ by putting $\lambda(0) = 0$.
Now, $\rho(0)\neq 0$, so we put $\rho(0) = 1$. Then, there are
two cases to consider.

\medskip
\noindent
\emph{Case 1.\/} $\lambda^{-1}(1)$ is not defined.
Then, if $J$ contains a nonempty subset $J'$ disconnected with $\{0,1\}$,
Theorem~\ref{si-kites} yields a contradiction. Thus,
$J = \{0,1\}$ and, since $\lambda$ and $\rho$ are injections,
$I = \{0\}$. The claim holds in this case.

\medskip
\noindent
\emph{Case 2.\/} $\lambda^{-1}(1)$ is defined. Then, we extend the numbering of
$J$ putting $\lambda^{-1}(1) = 1$. Since $J|<|I|$, there
is $j\in J\setminus\{0,1\}$. Now, if $\rho(1) = 0$, then $\{0,1\}$
and $j$ are disconnected, contradicting subdirect irreducibility.
Therefore, $\rho(1)\neq 0$, and by injectiveness $\rho(1)\neq 1$.
We then put $\rho(1) = 2$ extending the numbering of $I$.

Then, we repeat the procedure recursively, and by finiteness it must terminate.
By inspection of the two cases, it is clear that it terminates in Case~1,
for $j\in J$ such that $\lambda^{-1}(j)$ is not defined. Observe that
$j = n+1 = \rho(n)$, where $n$ was numbered at the immediately preceding stage
employing Case~2. This results in the required numbering of $I$ and $J$ and at
the same time shows that $|I| = |J|+1$.
\end{proof}

Kites of the form $K_{I,J}^{\lambda,\rho}(\mathbf{G})$ with
$I$ and $J$ finite will be called \emph{finite-dimensional}. We will deal with
these in more detail in the next section.
Next, we focus on subdirectly irreducible kites with $I$ or $J$ infinite.

\begin{lemma}\label{countable-dim}
Let  $K_{I,J}^{\lambda,\rho}(\mathbf{G})$ be a subdirectly
irreducible kite. Then, $I$ and $J$ are at most countably infinite.
\end{lemma}

\begin{proof}
Suppose $I$ is uncountable, but $J$ is countable. Then,
$\lambda(J)$ and $\rho(J)$ are also countable, and
therefore $I\setminus \bigl(\lambda(J)\cup\rho(J)\bigr)$ is nonempty.
Thus for any $k$ belonging to that set, we have
that for  $i = k = j$ neither of
the conditions stated in the second part of Theorem~\ref{si-kites}(2) can hold.

Suppose $J$ is uncountable, and thus so is $I$ because $\lambda$ and $\rho$ are
injections. Take $i\in \lambda(J)$. The set
$\bigcup_{m\in\omega}(\rho\circ\lambda^{-1})^m(i)$ is also countable
and therefore $I\setminus \bigcup_{m\in\omega}(\rho\circ\lambda^{-1})^m(i)$
is uncountable. If $i\notin\rho(J)$, then any pair
$(i, j)$ with $j\in I\setminus
\bigcup_{m\in\omega}(\rho\circ\lambda^{-1})^m(i)$
witnesses a failure of the conditions from the second part of
Theorem~\ref{si-kites}(2). If $i\in\rho(J)$, then,
since
$\bigcup_{m\in\omega}(\lambda\circ\rho^{-1})^m(i)$ is countable as well, we
have a $k\in I\setminus\bigcup_{m\in\omega}(\lambda\circ\rho^{-1})^m(i)$.
Then, the pair $(i, k)$ witnesses a failure of these conditions.
\end{proof}

\begin{lemma}\label{countable-dim-maps}
Let  $K_{I,J}^{\lambda,\rho}(\mathbf{G})$ be a subdirectly
irreducible kite with countably infinite $I$ and $J$.
Then, one of the following three cases must obtain:
\begin{enumerate}
\item $\lambda$ and $\rho$ are bijections.
\item $\lambda$ is a bijection and $|I\setminus\rho(J)| = 1$.
\item $\rho$ is a bijection and $|I\setminus\lambda(J)| = 1$.
\end{enumerate}
\end{lemma}

\begin{proof}
From the conditions in Theorem \ref{si-kites}(2) it follows that
$\rho(J)\cup\lambda(J) = I$. Suppose
$i,j\in\lambda(J)\setminus\rho(J)$. Observe that any path of
alternating $\lambda$ and $\rho$ between $i$ and $j$ must begin with
$\lambda^{-1}$. Therefore, it must end with $\rho$, and so $i = j$
and $m = 0$. So, $|\lambda(J)\setminus\rho(J)|\leq 1$. Similarly,
$|\rho(J)\setminus\lambda(J)|\leq 1$. We have three cases to
consider.

\medskip
\noindent
\emph{Case 1.\/} Suppose $|\lambda(J)\setminus\rho(J)| = 1$ and
$|\rho(J)\setminus\lambda(J)| = 1$. Put
$i = \lambda(J)\setminus\rho(J)$ and
$j = \rho(J)\setminus\lambda(J)$. Since $K_{I,J}^{\lambda,\rho}(\mathbf{G})$ is
subdirectly irreducible, there is an $m\in \omega$ such that
$(\rho\circ\lambda^{-1})^m(i) = j$. For $0\leq n\leq m$, put
$k_n = (\rho\circ\lambda^{-1})^n(i)$. Take
$k\in J\setminus \{k_0,\dots,k_m\}$ and $k_n$ with  $0\leq n \leq m$.
It is not difficult to show, by case analysis, that no path of alternating
$\lambda$ and $\rho$ can exist between $k$ and $k_n$, unless $i=j$. But then
either $|\lambda(J)\setminus\rho(J)| \neq 1$ or
$|\rho(J)\setminus\lambda(J)| \neq 1$, contradicting assumptions. This case
is thus excluded.

\medskip
\noindent
\emph{Case 2.\/} Suppose $|\lambda(J)\setminus\rho(J)| = 0$ and
$|\rho(J)\setminus\lambda(J)| = 1$. Then,
$\lambda(J)\subset\rho(J)$ and so $I = \lambda(J)\cup\rho(J) = \rho(J)$.
Thus, $\rho$ is a bijection and $|I\setminus\lambda(J)| = 1$, i.e.,
(3) holds. By symmetry, if $|\rho(J)\setminus\lambda(J)| = 0$ and
$|\lambda(J)\setminus\rho(J)| = 1$, we get that (2) holds.

\medskip
\noindent
\emph{Case 3.\/} Finally, suppose $|\lambda(J)\setminus\rho(J)| = 0 =
|\rho(J)\setminus\lambda(J)|$. Then $\lambda(J) = \rho(J) = I$ and
(1) holds.
\end{proof}

\subsection{Three infinite-dimensional kites}
It is consistent with Lemma~\ref{countable-dim-maps} that no
infinite-dimensional subdirectly irreducible kites exist. The
examples below show that cases (1), (2) and (3) in
Lemma~\ref{countable-dim-maps} are non-void.

Case (1). Take $I = J = \mathbb{Z}$ and put $\lambda(i) = i$,
$\rho(i) = i+1$. The kite
$K_{\mathbb{Z},\mathbb{Z}}^{\lambda,\rho}(\mathbb{Z})$ is
subdirectly irreducible; its smallest nontrivial normal filter is
the set of all sequences $\langle k_i\colon i\in\mathbb{Z}\rangle$
with $k_i \neq 0$ for finitely many $i$.

Case (2). Take $I = J = \omega$ and put $\lambda(i) = i$, $\rho(i) =
i+1$. The kite
$K_{\mathbb{Z},\mathbb{Z}}^{\lambda,\rho}(\mathbb{Z})$ is
subdirectly irreducible; its smallest nontrivial normal filter is
the set of all sequences $\langle k_i\colon i\in\omega\rangle$ with
$k_i \neq 0$ for finitely many $i$.

Case (3). Take $I = J = \omega$ and put $\lambda(i) = i+1$, $\rho(i)
= i$. We obtain an example which is symmetric, but not isomorphic,
to the previous one.

\subsection{Classification}
Gathering all results of this section together, we obtain the classification
of subdirectly irreducible kites promised at the beginning.

\begin{theorem}\label{classif}
Let $K_{I,J}^{\lambda,\rho}(\mathbf{G})$ be a subdirectly
irreducible kite. Then, $K_{I,J}^{\lambda,\rho}(\mathbf{G})$ is isomorphic to
precisely one of:
\begin{enumerate}
\item $K_{n,n}^{\lambda,\rho}(\mathbf{G})$, with $\lambda(j) = j$ and
$\rho(j) = j+1\ (\mathrm{mod}\ n)$.
\item $K_{\mathbb{Z},\mathbb{Z}}^{\lambda,\rho}(\mathbf{G})$, with $\lambda(j) = j$ and
$\rho(j) = j+1$.
\item $K_{\omega,\omega}^{\lambda,\rho}(\mathbf{G})$, with $\lambda(j) = j$ and
$\rho(j) = j+1$.
\item $K_{\omega,\omega}^{\lambda,\rho}(\mathbf{G})$, with $\lambda(j) = j+1$ and
$\rho(j) = j$.
\item $K_{n+1,n}^{\lambda,\rho}(\mathbf{G})$, with $\lambda(j) = j$ and
$\rho(j) = j+1$.
\end{enumerate}
Moreover, types {\rm (1)} and {\rm (2)} consist entirely of pseudo
MV-algebras, the other types contain no pseudo MV-algebras except
the two-element Boolean algebra. A kite of type {\rm (3)} or {\rm
(4)} is good if and only if it is a two-element Boolean algebra. A
kite of type {\rm (5)} is good if and only if $J=\emptyset$.
\end{theorem}

\begin{proof}
The cases with $I$ and $J$ finite follow from Lemma~\ref{si-sizes}. The
infinite-dimensional cases can be derived from Lemma~\ref{countable-dim-maps}
using arguments mimicking those of Lemma~\ref{si-sizes}.
The `moreover' statements follow from Lemma~\ref{good-kites}.
\end{proof}

For kites of the types from Theorem~\ref{classif} above
it will be convenient from now on to use the following notational convention:
\begin{enumerate}
\item $K_{n,n}^{0,1}(\mathbf{G})$
\item $K_{\mathbb{Z},\mathbb{Z}}^{0,1}(\mathbf{G})$
\item $K_{\omega,\omega}^{0,1}(\mathbf{G})$
\item $K_{\omega,\omega}^{1,0}(\mathbf{G})$
\item $K_{n+1,n}^{0,1}(\mathbf{G})$
\end{enumerate}
where the upper indices correspond to the functions $\lambda$ and
$\rho$ in such a way that whenever $\lambda$ or $\rho$ is the
identity function, we replace it by $0$; otherwise we replace it by
$1$ (which sits well with the fact that it is then a kind of
successor function).

\begin{cor}
A subdirectly irreducible kite is good if and only if it is either a
pseudo MV-algebra or it is of the form
$K_{1,0}^{\emptyset,\emptyset}(\mathbf{G})$, for a subdirectly
irreducible $\ell$-group $\mathbf{G}$.
\end{cor}

The last result in this section shows that subdirectly irreducible kites are
building blocks from which all kites can be built. Namely, we will show that
each kite is a subdirect product of subdirectly irreducible ones. This is not
a trivial corollary of Birkhoff's subdirect representation theorem, the value
added is that the subdirect factors are kites as well.

\begin{theorem}\label{si-kites-are-all}
Let $K_{I,J}^{\lambda,\rho}(\mathbf{G})$ be a kite. Then,
$K_{I,J}^{\lambda,\rho}(\mathbf{G})$ is subdirectly embeddable into
a product of kites of the form
$K_{I',J'}^{\lambda,\rho}(\mathbf{G})$, where $I'\cup J'$ is a
connected component of the graph of $\lambda\circ\rho^{-1}\cup
\rho\circ\lambda^{-1}$ paths in $I\cup J$.
\end{theorem}

\begin{proof}
It is not difficult to show, using Lemmas~\ref{hyper-filters}
and~\ref{restrict}, that if the $\ell$-group $\mathbf{G}$ is not
subdirectly irreducible, then any subdirect representation
$\mathbf{G}\leq \prod_{s\in  S}\mathbf{G}_s$ naturally gives rise to
a subdirect representation $K_{I,J}^{\lambda,\rho}(\mathbf{G})\leq
\prod_{s\in S}K_{I,J}^{\lambda,\rho}(\mathbf{G}_s)$. We leave the
details to the reader.

Now, assume $\mathbf{G}$ is subdirectly irreducible and let $P$
stand for the graph of $\lambda\circ\rho^{-1}\cup
\rho\circ\lambda^{-1}$ paths in $I\cup J$.  Let further
$\mathcal{C}$ be the set of all connected components of $P$. From
Theorem~\ref{si-kites} it follows that for each connected component
$C = I_C\cup J_C$ of $P$ the kite
$K_{I_C,J_C}^{\lambda,\rho}(\mathbf{G})$ is subdirectly irreducible.
For each $C\in\mathcal{C}$ let $N_C$ be the set of all $a = \langle
a^{-1}_i\in I\rangle\in (G^-)^I$ such that $i\notin C$ implies $a_i
= e$. It is straightforward to see that $N_C$ is a normal filter for
each $C\in\mathcal{C}$, and that
$K_{I,J}^{\lambda,\rho}(\mathbf{G})/\Theta_C$, where $\Theta_C$ is
the congruence corresponding to $N_C$, is isomorphic to
$K_{I_C,J_C}^{\lambda,\rho}(\mathbf{G})$. Since
$\bigcap_{C\in\mathcal{C}} N_C = \{1\}$, this proves claim.
\end{proof}

\begin{cor}\label{si-kites-generate}
Let\/ $\mathsf{K}$ be the variety generated by all kites.
Then, $\mathsf{K}$ is generated by all subdirectly irreducible kites.
\end{cor}

\section{Finite-dimensional kites}

A kite $K_{I,J}^{\lambda,\rho}(\mathbf{G})$ will be called
\emph{$n$-dimensional} if $|I| = n\in\omega$, and
\emph{finite-dimensional} if it is $n$-dimensional for some $n$. We
write $\mathcal{K}_n$ for the class of all $n$-dimensional kites and
$\mathsf{K_n}$ for the variety generated by $\mathcal{K}_n$. In this
section we will show that $\mathsf{K}$ is generated by
finite-dimensional kites, and thus $\mathsf{K}$ is the varietal join
of $\mathsf{K_n}$ for $n\in\omega$. Our proof proceeds by embedding
infinitely dimensional kites from Theorem~\ref{classif} into a
quotient of a subalgebra of a product of kites of finitely
dimensional kites.

For any $n\in \omega$, take the kite
$K_{2n+1,2n+1}^{0,1}(\mathbf{G})$. We think of the set $2n+1$ as the
universe of the additive group $\mathbb{Z}/(2n+1)\mathbb{Z}$ and
label its elements accordingly so that $2n+1 =
\{-n,-n+1,\dots,-1,0,1,\dots,n-1,n\}$. Consider the direct product
$\prod_{n\in\omega}K_{2n+1,2n+1}^{0,1}(\mathbf{G})$. As usual, for
$u\in\prod_{n\in\omega}K_{2n+1,2n+1}^{0,1}(\mathbf{G})$, we will
write $u(i)$ for the $i$-th element of $u$. Then, employing our
numbering convention, we will write $u(i)$ as $\langle
u_{-i},\dots,u_0,\dots,u_i\rangle$, and so we have
$$
u = \bigl\langle\langle u_{-n},\dots,u_0,\dots,u_n\rangle\colon n\in\omega\bigr\rangle.
$$
It is easy to see that the set $S_\mathbf{G} = U_\mathbf{G}\cup L_\mathbf{G}$,
where
\begin{align*}
U_\mathbf{G} &=  \bigl\{\langle a^{-1}_{-n},\dots,a^{-1}_i,\dots,a^{-1}_{n}\rangle,
\colon a^{-1}_i\in G^-,\ n\in\omega\bigr\}\\
L_\mathbf{G} &=  \bigl\{\langle f_{-n},\dots,f_i,\dots,f_{n}\rangle
\colon f_i\in G^+,\ n\in\omega\bigr\}
\end{align*}
is a subuniverse of $\prod_{n\in\omega}K_{2n+1,2n+1}^{0,1}(\mathbf{G})$.
Let $\mathbf{S}_\mathbf{G}$ be the corresponding subalgebra of
$\prod_{n\in\omega}K_{2n+1,2n+1}^{0,1}(\mathbf{G})$.
Now, on $\mathbf{S}_\mathbf{G}$ we define a binary
relation, putting
$$
u\sim w \quad\text{if{}f}\quad
\exists k\in\omega\ \forall n\geq k\colon [\![u(n)\neq w(n)]\!]\cap [-n+k,n-k] = \emptyset
$$
where $[\![u(n)\neq w(n)]\!] = \{-n\leq i\leq n\colon u_i\neq w_i\}$, as usual.
Intuitively, $u\sim w$ holds if, for sufficiently large $n\in\omega$,
the sequences $u(n)$ and $w(n)$ differ only at a bounded number of initial and
final places.  It should be clear that $u\sim w$ can hold only if either
$u,w\in U_\mathbf{G}$ or $u,w\in L_\mathbf{G}$.

\begin{lemma}\label{sim-is-cong}
The relation $\sim$ is a congruence on
$\prod_{n\in\omega}K_{2n+1,2n+1}^{0,1}(\mathbf{G})$.
\end{lemma}

\begin{proof}
Reflexivity and symmetry are obvious. For transitivity, suppose
$u\sim v$ and $v\sim w$. By definition then, there are
$k_1,k_2\in\omega$ such that
$$
\forall n\geq k_1\colon
[\![u(n)\neq v(n)]\!]\cap [-n+k_1,n-k_1] = \emptyset
$$
and
$$
\forall n\geq k_2\colon
[\![v(n)\neq w(n)]\!]\cap [-n+k_2,n-k_2] = \emptyset.
$$
Then, putting $k = \mathrm{max}\{k_1,k_2\}$ we obtain
$$
\forall n\geq k\colon
[\![u(n)\neq w(n)]\!]\cap [-n+k,n-k] = \emptyset
$$
and so $u\sim w$ as required. It remains to show that $\sim$ preserves the
operations. We will only prove it for two cases of multiplication.
Let $u\sim w$ and $v\sim s$, with $u,w\in L_\mathbf{G}$ and
$v,s\in U_\mathbf{G}$. Thus,
\begin{align*}
u(n) &= \langle f_{-n},\dots,f_i,\dots,f_{n}\rangle\\
w(n) &= \langle g_{-n},\dots,g_i,\dots,g_{n}\rangle\\
v(n) &= \langle a^{-1}_{-n},\dots,a^{-1}_i,\dots,a^{-1}_{n}\rangle\\
s(n) &= \langle b^{-1}_{-n},\dots,b^{-1}_i,\dots,b^{-1}_{n}\rangle
\end{align*}
with $a_i^{-1},b_i^{-1}\in G^-$ and $f_i,g_i\in G^+$,
for any $n\in\omega$. Further, by definition of $\sim$,
there are $k_1,k_2\in\omega$ such that
\begin{align*}
\forall n\geq k_1\colon
[\![u(n)\neq w(n)]\!]\cap [-n+k_1,n-k_1] &= \emptyset\\
\forall n\geq k_2\colon
[\![v(n)\neq s(n)]\!]\cap [-n+k_2,n-k_2] &= \emptyset.
\end{align*}
Consider $uv$ and $ws$.
\begin{align*}
uv &=
\langle u(n)\colon n\in\omega\rangle\cdot\langle v(n)\colon n\in\omega\rangle\\
 &= \langle u(n)\cdot v(n)\colon n\in\omega\rangle\\
 &= \bigl\langle\langle f_{-n},\dots,f_n\rangle\cdot
\langle a^{-1}_{-n},\dots,a^{-1}_n\rangle\colon n\in\omega\bigr\rangle\\
 &= \bigl\langle\langle f_{-n}a^{-1}_{-n+1}\vee e,\dots,
f_{n-1}a^{-1}_n\vee e,f_na^{-1}_{{n+1}(\mathrm{mod}\ 2n+1)}\vee e\rangle\colon n\in\omega\bigr\rangle\\
 &= \bigl\langle\langle f_{-n}a^{-1}_{-n+1}\vee e,\dots,
f_{n-1}a^{-1}_n\vee e,f_na^{-1}_{-n}\vee e\rangle\colon n\in\omega\bigr\rangle
\end{align*}
and similarly,
\begin{align*}
ws &= \langle w(n)\cdot s(n)\colon n\in\omega\rangle\\
 &= \bigl\langle\langle g_{-n}b^{-1}_{-n+1}\vee e,\dots,
g_{n-1}b^{-1}_n\vee e,g_nb^{-1}_{-n}\vee e\rangle\colon n\in\omega\bigr\rangle.
\end{align*}
Now, for a given $n\geq k = \mathrm{max}\{k_1,k_2\}$, let us compare
\begin{align*}
u(n)\cdot v(n) &=
\langle f_{-n}a^{-1}_{-n+1}\vee e,\dots,f_na^{-1}_{-n}\vee e\rangle\\
w(n)\cdot s(n) &=
\langle g_{-n}b^{-1}_{-n+1}\vee e,\dots,g_nb^{-1}_{-n}\vee e\rangle.
\end{align*}
From the definition of multiplication on
$K_{2n+1,2n+1}^{0,1}(\mathbf{G})$ it follows immediately that the
two can differ only at initial places up to
$f_{-n+k}a^{-1}_{-n+k+1}$, $g_{-n+k}b^{-1}_{-n+k+1}$, respectively,
and at final places from $f_{n-k-1}a^{-1}_{-n-k}$,
$g_{n-k-1}b^{-1}_{-n-k}$, respectively. It follows that
$$
\forall n\geq k+1\colon
[\![u(n)\cdot v(n)\neq w(n)\cdot s(n)]\!]\cap [-n+k+1,n-k-1] =\emptyset
$$
and thus $uv\sim ws$ as required. Other cases of multiplication are similar, and
analogous arguments prove that $\sim$ preserves all the operations, completing
the proof.
\end{proof}

We are now ready to show that $K_{\mathbb{Z},\mathbb{Z}}^{0,1}(\mathbf{G})$
embeds into $\mathbf{S}_\mathbf{G}/\kern-3pt\sim$. To this end, first
define a map
$\mu^-\colon K_{\mathbb{Z},\mathbb{Z}}^{0,1}(\mathbf{G})
\to \mathbf{S}_\mathbf{G}$ putting
$$
\langle u_i\colon i\in\mathbb{Z}\rangle \mapsto
\bigl\langle\langle u_{-n},\dots, u_0,\dots, u_{n}\rangle, n\in\omega\bigr\rangle
$$
where $\langle u_i\colon i\in\mathbb{Z}\rangle$ is either
$\langle a_i^{-1}\colon i\in\mathbb{Z}\rangle$ with $a_i\in G^{-}$ for each $i$,
or $\langle f_i\colon i\in\mathbb{Z}\rangle$ with $f_i\in G^{+}$ for each $i$.
The map $\mu^-$ extends naturally to a map
$\mu\colon K_{\mathbb{Z},\mathbb{Z}}^{0,1}(\mathbf{G})
\to \mathbf{S}_{\mathbf{G}}/\kern-3pt\sim$, namely $\mu = \mu^-/\kern-3pt\sim$.

\begin{lemma}\label{embed-Z}
The map $\mu$ above is an embedding of
$K_{\mathbb{Z},\mathbb{Z}}^{0,1}(\mathbf{G})$ into
$\mathbf{S}_{\mathbf{G}}/\kern-3pt\sim$.
\end{lemma}

\begin{proof}
To show that $\mu$ is one-one, let $u = \langle u_i\colon
i\in\mathbb{Z}\rangle$ and $w = \langle w_i\colon
i\in\mathbb{Z}\rangle$ be such that $u_i= w_i$ for all $i\neq \ell$
and $u_\ell \neq w_\ell$, where $\ell\in\mathbb{Z}$ is arbitrary,
but fixed. By definition, $\mu(u)\neq\mu(w)$ if{}f
$\mu^-(u)\not\sim\mu^-(w)$. Now, observe that for any $k\in \omega$,
if we take $n\geq k+|\ell|+1$, then $-n+k<\ell<n-k$, and therefore
$\ell\in [\![u(n)\neq w(n)[\!]$. It follows that
$$
\forall k\ \exists n\geq k\colon
[\![u(n)\neq w(n)[\!]\cap[-n+k,n-k]\neq\emptyset
$$
and so $\mu^-(u)\not\sim\mu^-(w)$ as needed.

To prove that $\mu$ is a homomorphism we proceed case by case. We only
present two cases of multiplication here. Let
$u = \langle a_i^{-1}\colon i\in \mathbb{Z}\rangle$ and
$w = \langle f_i\colon i\in \mathbb{Z}\rangle$. Then,
$uw = \langle a_i^{-1}f_i\vee e\colon i\in \mathbb{Z}\rangle$ and
$wu = \langle f_ia_{i+1}^{-1}\vee e\colon i\in \mathbb{Z}\rangle$.
Thus,
$$
\mu^-(uw) =
\bigl\langle\langle a_{-n}^{-1}f_{-n}\vee e,\dots,
a_0^{-1}f_0\vee e,\dots, a_{n}^{-1}f_{n}\vee e\rangle, n\in\omega\bigr\rangle
$$
and
$$
\mu^-(wu) =
\bigl\langle\langle f_{-n}a_{-n+1}^{-1}\vee e,\dots,
f_0a_1^{-1}\vee e,\dots, f_{n}a_{n+1}^{-1}\vee e\rangle, n\in\omega\bigr\rangle
$$
On the other hand, we have
\begin{align*}
\mu^-(u)\mu^-(w) &=
\bigl\langle\langle a_{-n}^{-1},\dots, a_0^{-1},\dots, a_{n}^{-1}\rangle,
n\in\omega\bigr\rangle\cdot
\bigl\langle\langle f_{-n},\dots, f_0,\dots, f_{n}\rangle,
n\in\omega\bigr\rangle\\
&= \bigl\langle\langle a_{-n}^{-1}f_{-n}\vee e,\dots, a_0^{-1}f_0\vee e,
\dots, a_{n}^{-1}f_n\vee e\rangle, n\in\omega\bigr\rangle
\end{align*}
and also,
\begin{align*}
\mu^-(w)\mu^-(u) &= \bigl\langle\langle f_{-n},\dots, f_0,\dots,
f_{n}\rangle, n\in\omega\bigr\rangle\cdot \bigl\langle\langle
a_{-n}^{-1},\dots, a_0^{-1},\dots, a_{n}^{-1}\rangle,
n\in\omega\bigr\rangle\\
&= \bigl\langle\langle f_{-n}a_{-n+1}^{-1}\vee e,\dots, f_0a_1^{-1}\vee e,
\dots, f_na_{{n+1}(\mathrm{mod}\ 2n+1)}^{-1}\vee e\rangle, n\in\omega\bigr\rangle\\
&= \bigl\langle\langle f_{-n}a_{-n+1}^{-1}\vee e,\dots, f_0a_1^{-1}\vee e,
\dots, f_na_{-n}^{-1}\vee e\rangle, n\in\omega\bigr\rangle.
\end{align*}
Thus, for every $n\in\omega$ we have that $\bigl(\mu^-(uw)\bigr)(n)
= \bigl(\mu^-(u)\mu^-(w)\bigr)(n)$ and the sequences
$\bigl(\mu^-(wu)\bigr)(n)$ and $\bigl(\mu^-(w)\mu^-(u)\bigr)(n)$
differ at most at the final place. It follows that
$$
[\![\bigl(\mu^-(uw)\bigr)(n)\neq
\bigl(\mu^-(u)\mu^-(w)\bigr)(n)]\!]\cap[-n+1,n-1] = \emptyset
$$
and
$$
[\![\bigl(\mu^-(wu)\bigr)(n)\neq
\bigl(\mu^-(w)\mu^-(u)\bigr)(n)]\!]\cap[-n+1,n-1] = \emptyset
$$
proving $\mu^-(uw)\sim \mu^-(u)\mu^-(w)$ and $\mu^-(wu)\sim
\mu^-(w)\mu^-(u)$, as required.
\end{proof}

Next we turn to $K_{\omega,\omega}^{0,1}(\mathbf{G})$. The argument
is by and large analogous to the one just used, but simpler. Take
the direct product $\prod_{n\in\omega}K_{n+1,n}^{0,1}(\mathbf{G})$.
We have that $S'_\mathbf{G} = U'_\mathbf{G}\cup L'_\mathbf{G}$,
where
\begin{align*}
U'_\mathbf{G} &=  \bigl\{\langle
a^{-1}_0,\dots,a^{-1}_i,\dots,a^{-1}_{n+1}\rangle,
\colon a^{-1}_i\in G^-,\ n\in\omega\bigr\}\\
L'_\mathbf{G} &=  \bigl\{\langle f_0,\dots,f_i,\dots,f_{n}\rangle
\colon f_i\in G^+,\ n\in\omega\bigr\}
\end{align*}
is a subuniverse of $\prod_{n\in\omega}K_{n+1,n}^{0,1}(\mathbf{G})$.
Let $\mathbf{S}'_\mathbf{G}$ be the corresponding subalgebra of
$\prod_{n\in\omega}K_{n+1,n}^{0,1}(\mathbf{G})$. Here we can embed
$K_{\omega,\omega}^{0,1}(\mathbf{G})$ directly into $\mathbf{S}'_\mathbf{G}$.
Define a map
$\nu\colon K_{\omega,\omega}^{0,1}(\mathbf{G})
\to \mathbf{S}'_{\mathbf{G}}$
putting
$$
\langle f_i\colon i\in\omega\rangle \mapsto
\bigl\langle\langle f_0,\dots, f_n\rangle, n\in\omega\bigr\rangle
$$
where $f_i\in G^+$ for all $i$, and
$$
\langle a_i^{-1}\colon i\in\omega\rangle \mapsto \bigl\langle\langle
a_0^{-1}, a_1^{-1}, \dots, a_{n+1}^{-1}\rangle,
n\in\omega\bigr\rangle
$$
where $a_i^{-1}\in G^-$ for all $i$. It is then rather straightforward to prove
the following lemma.

\begin{lemma}\label{embed-omega-case1}
The map $\nu$ above is an embedding of $K_{\omega,\omega}^{0,1}(\mathbf{G})$
into $\mathbf{S}'_{\mathbf{G}}$.
\end{lemma}

Finally, we deal with $K_{\omega,\omega}^{1,0}(\mathbf{G})$. We
again use the algebra $\mathbf{S}''_\mathbf{G}$, but this time
define a map $\nu'\colon K_{\omega,\omega}^{1,0}(\mathbf{G}) \to
\mathbf{S}'_{\mathbf{G}}$ by reversing the ordering in $\nu$,
namely, by
$$
\langle f_i\colon i\in\omega\rangle \mapsto
\bigl\langle\langle f_n,f_{n-1},\dots, f_0\rangle, n\in\omega\bigr\rangle
$$
where $f_i\in G^+$ for all $i$, and
$$
\langle a_i^{-1}\colon i\in\omega\rangle \mapsto \bigl\langle\langle
a_{n+1}^{-1}, a_n^{-1}, \dots, a_0^{-1}\rangle,
n\in\omega\bigr\rangle
$$
where $a_i^{-1}\in G^-$ for all $i$. We obtain the lemma below.

\begin{lemma}\label{embed-omega-case2}
The map $\nu'$ above is an embedding of $K_{\omega,\omega}^{0,1}(\mathbf{G})$
into $\mathbf{S}'_{\mathbf{G}}$.
\end{lemma}

\begin{theorem}\label{generation}
The variety $\mathsf{K}$ is generated by all finite-dimensional kites.
\end{theorem}

\begin{proof}
By Corollary~\ref{si-kites-generate}, the variety
$\mathsf{K}$ is generated by all subdirectly irreducible kites.
Let $\mathcal{K}_{\mathrm{fin}}$ be the class of all finitely dimensional kites.
By Lemmas~\ref{embed-Z}, \ref{embed-omega-case1}, and~\ref{embed-omega-case2},
every subdirectly irreducible kite belongs to
$SHSP(\mathcal{K}_{\mathrm{fin}})$. Thus $\mathsf{K} =
V(\mathcal{K}_{\mathrm{fin}})$ as claimed.
\end{proof}

\begin{cor}
The variety $\mathsf{K}$ generated by all kites is the varietal join of
varieties $\mathsf{K}_n$, generated by $n$-dimensional kites. Briefly,
$$
\mathsf{K} = \bigvee_{n\in\omega}\mathsf{K}_n.
$$
\end{cor}

\section{Applications}\label{appl}

Let $\mathbf{A}=(A;\ld,\rd, \cdot, 0,1)$ be a pseudo BL-algebra and
let $g\in A \setminus \{1\}.$ Then there is a  filter, $V,$ of $A$
not containing $g$ and is maximal with respect to this property.  We
call it a \emph{value} of $g,$ and the filter $V^*$ generated by $V$
and the element $g$ is said to be a \emph{cover} of $V$. As in
$\ell$-groups, we say that $\mathbf{A}$ is \emph{normal-valued} if
every value is normal in its cover. Let $\mathsf{NVpsBL}$ be the
class of normal-valued pseudo BL-algebras. In \cite{DGK} it was
proved that $\mathsf{NVpsBL}$ is a variety, and in contrast to the
variety of $\ell$-groups, it is not the greatest proper subvariety
of the variety of $\ell$-groups because $\mathsf{NVpsBL}$ is a
proper subvariety of the system of all pseudo BL-algebras
$\mathbf{A}$ such that every maximal filter of $\mathbf A$ is normal
that is also a variety, see \cite{DGK}.

According to Wolfenstein~\cite[Thm 41.1]{Dar}, an $\ell$-group $G$
is normal-valued if{}f  every $a,b \in G^-$ satisfy $a^2b^2 \le ba$,
or in our language
$$
a^2\cdot b^2 \le b\cdot a.\eqno(1)
$$
In \cite{BDK} it was shown a large variety of pseudo BL-algebras
that are normal-valued iff they satisfy (1). We note that the same
is true for kites.

\begin{theorem}\label{norm-val}
Let $\mathbf{G}$ be an $\ell$-group. Then the kite
$K_{I,J}^{\lambda,\rho}(\mathbf{G})$ is normal-valued if and only if
{\rm (1)} holds for all $a,b \in
K_{I,J}^{\lambda,\rho}(\mathbf{G}).$
\end{theorem}

\begin{proof}
It follows from the definition of kites and Wolfenstein's
criterion,  \cite[Thm 41.1]{Dar}.
\end{proof}

We note that if $x \in K_{I,J}^{\lambda,\rho}(\mathbf{G}),$ then
$$
x^2 = 0 \quad \mbox{or}\quad  (\ln x)^2 =0. \eqno(2)
$$
Let $\mathsf{V}(K_{I,J}^{\lambda,\rho}(\mathbf{G}))$ denote the
variety of pseudo BL-algebras generated by a kite
$K_{I,J}^{\lambda,\rho}(\mathbf{G})$.


Let $n\ge 2$ be a fixed integer and we set $J_n=\{1,\ldots,n\},$
$I_n=\{1,\ldots,n,n+1\},$ $J_n'=\{2,\ldots,n\},$ and
$I_n'=\{2,\ldots,n,n+1\}.$ We set $\lambda,\rho: J_n\to I_n$ by
$\lambda(i)=i$ and $\rho(i)=i+1$ for each $i \in J_n.$ Take two
terms $f_\sim(x):= \ln x$ and $f_-(x)=\rn x$.  Then on
$K_{I_n,J_n}^{\lambda_n,\rho_n}(\mathbf{G})$ we have
$$
f^{2n+1}_\sim(x) \in \{0,1\} =
B(K_{I_n,J_n}^{\lambda_n,\rho_n}(\mathbf{G}))\eqno(3)
$$
for every $x\in K_{I_n,J_n}^{\lambda_n,\rho_n}(\mathbf{G}).$

For abbreviation, given an integer $n\ge 1,$ we set $\mathbb
Z_n^\dag = K_{I_n,J_n}^{\lambda_n,\rho_n}(\mathbf{Z}).$ In addition,
we define $\mathbb Z_0^\dag:=
K_{I_1,J_1}^{\lambda_1,\rho_1}(\mathbf{O}),$ where $\mathbf O$ is
the trivial $\ell$-group consisting only from the identity. Then
$\mathbb Z_0^\dag$ is the two-element Boolean algebra, therefore,
$\mathbb Z_0^\dag$ generates the variety  of Boolean algebras, $\mathsf{BA}.$

\begin{theorem}\label{covers}
For any integer $n \ge 1$,  ${\mathsf{V}}(\mathbb Z_{n}^\dag)$ is
a  cover of the variety $\mathsf{BA},$ and $n\ne m$ implies
${\mathsf{V}}(\mathbb Z_{n}^\dag) \ne {\mathsf{V}}(\mathbb Z_{m}^\dag).$
\end{theorem}

\begin{proof}
For $n = 0$, the variety ${\mathsf{V}}(\mathbb Z_{n}^\dag)$ is
precisely the variety of \emph{product logic algebras} and it is well-known that
this variety covers $\mathsf{BA}$.
For $n = 1$, it was proved in
\cite[Thm 11]{JiMo} that the variety $\mathsf V(\mathbb Z^\dag_1)$
covers $\mathsf{BA}$.  Thus, we can assume that $n \ge 1$.

It follows from Theorem~\ref{classif} that $\mathbb Z_n^\dag$ is
subdirectly irreducible.

\medskip
\emph{Claim 1. Any nontrivial element of $\mathbb
Z^\dag_n$ generates a subalgebra of $\mathbb Z^\dag_n$ that is an
isomorphic copy of $\mathbb Z^\dag_n.$}

\medskip
Assume $a=\langle a_1^{-1},\ldots,a_{n+1}^{-1}\rangle$ is an element
from $(G^-)^I$ and let $A$ be the subalgebra of $\mathbb Z^\dag_n$
generated by $a$. Let $i_0$ be the first index such that is
different of $e^{-1}$. There is an integer $n_0$ such that
$f_\sim^{n_0}(a)=\langle e^{-1},\ldots, e^{-1},a_{i_0}^{-1} \rangle
\in A$. There is also an integer $m$ such that $f_-^m(\langle
e^{-1},\ldots, e^{-1},a_{i_0}^{-1} \rangle) =\langle
a_{i_0},e^{-1},\ldots,e^{-1}\rangle \in A.$ In addition, for every
$i=1,\ldots, n+1,$ the element $x_i= \langle
b_1^{-1},\ldots,b_{n+1}^{-1}\rangle$ belongs to $A,$ where $b_i^{-1}
= a_{i_0}^{-1}$ and $b_{j}^{-1}=e^{-1}$ for $j\ne i.$

Since $\ln\langle a_1^{-1},\ldots,a^{-1}_{n+1}\rangle =
\ln\langle a_1^{-1},\ldots,a_n^{-1}\rangle, $ $\langle f_1,\ldots,f_n\rangle
= \langle e^{-1},f_1^{-1},\ldots,f_n^{-1}\rangle$, we see that
the algebra $A$ can be generated equivalently by some appropriate
element from $(G^{-})^I$ or from $(G^+)^J$.

Hence, if $a=\mbox{g.c.d.}\{a_1,\ldots,a_{n+1}\},$    we see that
$A$ is generated e.g. by the element $\langle a^{-1},e^{-1},\ldots,
e^{-1}\rangle,$ and it is an isomorphic copy of $\mathbb Z^\dag_n.$

\medskip
\emph{Claim 2.  If $\mathbf{A}$ is an algebra from
$\mathsf{V}(\mathbb Z^\dag_n)$, then for any $x\in A$, we have
$$
f_\sim^{2n+1}(x) \in B(A).\eqno(4)
$$}

\medskip
Indeed, let $C$ be a subdirectly irreducible algebra
from $\mathsf{V}(\mathbb Z^\dag_n).$  By the J\'onsson Lemma, $C$ is
a homomorphic image of a subalgebra $D$ of an ultrapower $\mathcal
U$ of   $\mathbb Z^\dag_n$ on the index set $U.$ By (3), if $x \in
\mathbb Z^\dag_n,$ then $f_\sim^{2n+1}(x)\in \{0,1\}= B(\mathbb
Z^\dag_n).$ If $x \in \langle t_u\colon u \in U\rangle, $ then
$f_\sim^{2n+1}(x) \in B((\mathbb Z^\dag_n)^U).$ Similarly, if $x \in
\langle t_u \colon u\in U\rangle/\mathcal U,$ $f^{2n+1}(x) \in
B(\mathbb (\mathbb Z^\dag_n)^U/\mathcal U).$ In the same way, we can
prove that if $x \in C,$ then $f^{2n+1}_\sim(x) \in B(C).$ The
general case follows from the statement follows from the fact that
$A$ is isomorphic to  a subdirect product of algebras from $\mathit
{HSP}_U(\mathbb Z^\dag_n).$

\medskip
\emph{Claim 3. Any subdirectly irreducible algebra
of $\mathsf{V}(\mathbb Z^\dag_n)$ is either the two-element Boolean
algebra or has a subalgebra isomorphic to $\mathbb Z^\dag_n.$}

\medskip
Let $C$ be a subdirectly irreducible algebra of
$\mathsf{V}(\mathbb Z^\dag_n)$ and assume that it has a
nontrivial element $0<c<1.$ Now we use again the J\'onsson Lemma,
and we assume that $C$ is a homomorphic image of a subalgebra $D$ of
an ultrapower $\mathcal U$ of   $\mathbb Z^\dag_n$ on the index set
$U.$ Let $d\in D$ be a preimage of $c$ under the homomorphism. We
show that $a=d^\sim$ generates a subalgebra of $D$ that is
isomorphic to $\mathbb Z^\dag_n.$ Then $a = \langle a_u: u \in
U\rangle/\mathcal U.$ Consider the set $V =\{u\in U: a^2_u =0\}$. If
$V \in \mathcal U,$ then $a$ generates in the same way as $\langle
1,\ldots, 1\rangle \in \mathbb (\mathbb Z^+)^n.$ If $V \notin \mathcal
U,$ we can use the generator like $\langle 0,\ldots,0,-1\rangle
\in \mathbb (\mathbb Z^-)^{n+1}.$

\medskip
\emph{Claim 4. Every $\mathsf{V}(\mathbb Z^\dag_{n})$
is a cover variety of the variety of Boolean algebras.}

\medskip
Let $\mathcal V$ be a subvariety of $\mathsf{V}(\mathbb Z^\dag_{n})$
containing properly the variety of Boolean algebras, and let $A$ be
a subdirectly irreducible algebra from $\mathcal V.$ By Claim 3, $A$
is either a two-element Boolean algebra or  contains an isomorphic
copy of $\mathbb Z^\dag_{n}.$  Therefore, in the later case, $A
\notin \mathsf{BA}$ and we have $\mathsf V(A) = \mathsf{V}(\mathbb
Z^\dag_{n})\subseteq \mathcal V$, which proves the statement of the
Theorem.

\medskip
\emph{Claim 5. If $n$ and $m$ are two different positive integers, then
$\mathsf{V}(\mathbb Z^\dag_{n}) \ne \mathsf{V}(\mathbb Z^\dag_{m}).$}

\medskip
Assume $n<m$.  Verifying  (4), we see that the
element $x=\langle-1,\ldots,-1 \rangle$ from $\mathbb Z^\dag_{m}$
does not belong to $\mathsf{V}(\mathbb Z^\dag_{n}).$
\end{proof}

The next result follows immediately from Lemma~\ref{good-kites},
Theorem~\ref{covers}, and the subdirect representation theorem.

\begin{cor}
An algebra $\mathbf{A}\in \mathsf V(\mathbb{Z}^\dag_n)$, for any $n\in\omega$,
is good if and only if $\mathbf{A}$ is a Boolean algebra.
\end{cor}

\end{document}